\documentclass[a4paper,12pt]{amsart} 

\usepackage{amssymb,amsmath}

\newtheorem{thm}{Theorem}[section] 
\newtheorem{prop}[thm]{Proposition}
\newtheorem{cor}[thm]{Corollary} 
\newtheorem{lem}[thm]{Lemma}

\theoremstyle{definition} 
 
\newtheorem{rem}[thm]{Remark}

\numberwithin{equation}{section}

\begin{document}
\baselineskip=17pt

\title{On geometry of frame bundles}
\author{Kamil Niedzia\l omski}

\thanks{The author is supported by the Polish NSC grant N6065/B/H03/2011/40}
\date{}

\address{
Department of Mathematics and Computer Science \endgraf
University of \L\'{o}d\'{z} \endgraf
ul. Banacha 22, 90-238 \L\'{o}d\'{z} \endgraf
Poland
}
\email{kamiln@math.uni.lodz.pl}

\begin{abstract}
Let $(M,g)$ be a Riemannian manifold, $L(M)$ its frame bundle. We construct new examples of Riemannian metrics, which are obtained from Riemaniann metrics on the tangent bundle $TM$. We compute the Levi--Civita connection and curvatures of these metrics.
\end{abstract}

\maketitle

\section{Introduction} 
Let $(M,g)$ be a Riemannian manifold, $L(M)$ its frame bundle. The first example of a Riemannian metric on $L(M)$ was considered by Mok \cite{mok}. This metric, called the Sasaki--Mok metric or the diagonal lift $g^d$ of $g$, was also investigated in \cite{cl1} and \cite{cl2}. It is very rigid, for example, $(L(M),g^d)$ is never locally symmetric unless $(M,g)$ is locally Euclidean. Moreover,  with respect to the Sasaki--Mok metric vertical and horizontal distributions are orthogonal. A wider and less rigid class of metrics $\bar g$, in which vertical and horizontal distributions are no longer orthogonal, has been recently considered by Kowalski and Sekizawa in the series of papers \cite{ks0,ks1,ks2}. These metrics are defined with respect to the decomposition of the vertical distribution $\mathcal{V}$ into $n=\dim M$ subdistributions $\mathcal{V}^1,\ldots,\mathcal{V}^n$. 

In this short paper we introduce a new class of Riemannian metrics on the frame bundle. We identify distributions $\mathcal{V}^i$ with the vertical distribution in the second tangent bundle $TTM$. Namely, each map $R_i:L(M)\to TM$, $R_i(u_1,\ldots,u_n)=u_i$ induces a linear isomorphism $R_{i\ast}:\mathcal{H}\oplus\mathcal{V}^i\to TTM$, where $\mathcal{H}$ is a horizontal distribution defined by the Levi--Civita connection $\nabla$ on $M$. By this identification we pull--back the Riemannian metric from $TM$. We pull--back natural metrics, in the sence of Kowalski and Sekizawa \cite{ks}, from $TM$ and study the geometry of such Riemannian manifolds. We compute the Levi--Civita connection, the curvature tensor, sectional and scalar curvature.

\section{Riemannian metrics on frame bundles}
Let $(M,g)$ be a Riemannian manifold. Its frame bundle $L(M)$ consists of pairs $(x,u)$ where $x=\pi_{L(M)}(u)\in M$ and $u=(u_1,\ldots,u_n)$ is a basis of a tangent space $T_xM$. We will write $u$ instead of $(x,u)$. Let $(x_1,\ldots,x_n)$ be a local coordinate system on $M$. Then, for every $i=1,\ldots,n$ we have
\[
u_i=\sum_j u^j_i\frac{\partial}{\partial x_j}
\]  
for some smooth functions $u^j_i$ on $L(M)$. Putting $\alpha_i=x_i\circ\pi_{L(M)}$, $(\alpha_i,u^j_k)$ is a local coordinate system on $L(M)$. Let $\omega$ be a connection form of $L(M)$ corresponding to Levi--Civita connection $\nabla$ on $M$. We have a decomposition of the tangent bundle $TL(M)$ into the {\it horizontal} and {\it vertical} distribution:
\[
T_uL(M)=\mathcal{H}^{L(M)}_u\oplus \mathcal{V}^{L(M)}_u,
\]
where $\mathcal{H}^{L(M)}={\rm ker}\omega$ and $\mathcal{V}^{L(M)}={\rm ker}\pi_{L(M)\ast}$. Let $X^h$ denotes the horizontal lift of a vector field $X$ on $M$. 

Decompose the second tangent bundle $TTM$ into horizontal and vertical part, $T_{\zeta}TM=\mathcal{H}^{TM}_{\zeta}\oplus\mathcal{V}^{TM}_{\zeta}$, with respect to the connection map $K:TTM\to TM$ and the projection in the tangent bundle $\pi_{TM}:TM\to M$, see for example \cite{dom}. Let $X^{h,TM}$ and $X^{v,TM}$ denote the horizontal and vertical lifts to $TTM$ of a vector field $X$ on $M$.   

For an index $i=1,\ldots,n$ define a map $R_i:L(M)\to TM$ as follows
\begin{equation*}
R_i(u)=u_i,\quad u=(u_1,\ldots,u_n)\in L(M).
\end{equation*}
$R_i$ is the right multiplication by a $i$--th vector of a canonical basis in $\mathbb{R}^n$.
\begin{prop} \label{p1}
 The operator $R_i$ has the following properties.
\begin{enumerate}
\item[(1)] We have
\begin{equation*}
R_{\xi\ast}X^h=X^{h,TM}.
\end{equation*}
In particular, $R_{i\ast}$ is an isomorphism of $\mathcal{H}^{L(M)}$ and $\mathcal{H}^{TM}$, 
\item[$(2)$] Let $\mathcal{V}^i$ be a linear subspace of $\mathcal{V}^{L(M)}$ spanned by fundamental vertical vectors $A^*$, where the matrix $A\in{\rm gl}(n)$ has only nonzero $i$--th column. Then $R_{i\ast}$ is an isomorphism of $\mathcal{V}^i$ and $\mathcal{V}^{TM}$, and is zero on $\mathcal{V}^j$ for $j\not=i$. Moreover, there is a decomposition
\begin{equation*}
\mathcal{V}^{L(M)}=\mathcal{V}^1\oplus\ldots\oplus\mathcal{V}^n.
\end{equation*}
\end{enumerate}
\end{prop}
\begin{proof}
Easy computations left to the reader.
\end{proof}
By Proposition \ref{p1}, we have natural identifications
\begin{gather} \label{ident1}
\begin{array}{ccccc}
\mathcal{H}^{L(M)} & \longleftrightarrow & \mathcal{H}^{TM} & \longleftrightarrow & TM \\
X^h & \longleftrightarrow & X^{h,TM} & \longleftrightarrow & X
\end{array}
\end{gather}
and
\begin{gather} \label{ident2}
\begin{array}{ccccc}
\mathcal{V}^i & \longleftrightarrow & \mathcal{V}^{TM} & \longleftrightarrow & TM \\
X^{v,i} & \longleftrightarrow & X^{v,TM} & \longleftrightarrow & X
\end{array}
\end{gather}
Hence, we have defined the vertical lift $X^{v,i}\in \mathcal{V}^i$ of the vector $X\in TM$ satisfying the property
\begin{equation*}
R_{i\ast}X^{v,i}=X^{v,TM}.
\end{equation*}

Let $c=(c_1,\ldots,c_n)\in\mathbb{R}^n$ and $C=(c_{ij})$ be $n\times n$ matrix. We assume that the $(n+1)\times (n+1)$ matrix
\begin{equation*}
\bar C=\left( \begin{array}{cc} 1 & c \\ c^{\top} & C \end{array} \right)
\end{equation*}
is symmetric and positive definite. Let $g_{TM}$ be a Riemannian metric on $TM$. 

Now, we are able to define a new class of Riemannian metrics $\bar g=\bar g_{\varepsilon,\bar C}$ on $L(M)$. Let $F:L(M)\to TM$ be any smooth function. Put
\begin{align*}
&\bar g(X^h,Y^h)_u =g_{TM}(X^{h,TM},Y^{h,TM})_{F(u)}, \\
&\bar g(X^h,Y^{v,i})_u =c_i g_{TM}(X^{h,TM},Y^{v,TM})_{F(u)}, \\
&\bar g(X^{v,i},Y^{v,j})_u =c_{ij} g_{TM}(X^{v,TM},Y^{v,TM})_{F(u)}.
\end{align*} 

Fix $u\in L(M)$. Let $e_1,\ldots,e_n$ be a basis in $T_xM$, $\pi(u)=x$, such that $(e_1)^{h,TM}_{F(u)},\ldots,(e_1)^{h,TM}_{F(u)}$ is an orthonormal basis in $\mathcal{H}^{TM}_{F(u)}$. Then
\begin{equation}\label{basis}
e_1^h,\ldots,e_n^h,e_1^{v,1},\ldots,e_n^{v,1},
\ldots,e_1^{v,n},\ldots,e_n^{v,n}
\end{equation}
is a basis in $T_uL(M)$. Let $G$ be a matrix of the Riemanian metric $g_{TM}$ with respect to the basis $e_1^{h,TM},\ldots,e_n^{h,TM},e_1^{v,TM},\ldots,e_n^{v,TM}$. The fact that $\bar g$ is positive definite follows from the following lemma.

\begin{lem}
Let
\begin{equation*}
G=\left( \begin{array}{cc} I & g^{hv} \\ g^{vh} & \hat g \end{array} \right)
\end{equation*}
be a positive definite symmetric $2n\times 2n$ block martix. Then the matrix
\begin{equation*}
\bar G=\left( \begin{array}{cc} I & c\otimes g^{vh} \\ c^{\top}\otimes g^{hv} & C\otimes \hat g \end{array} \right)
\end{equation*}
is positive definite.
\end{lem} 

\begin{proof}
It sufficies to show that each principal minor $\bar G_k$, $k=1,\ldots,n+n^2$, of $\bar G$ is positive. Obviously $\det\bar G_k=1>0$ for $k=1,\ldots,n$. Hence we assume $k>n$. Then each minor $\bar G_k$ is of the same form as the whole matrix $\bar G$, thus we will make calculations using matrix $\bar G$. Computing the determinant of the block matrix we get
\begin{align*}
\det \bar G &=\det(C\otimes \hat g-(c^{\top}\otimes g^{vh})(c\otimes g^{hv})) \\
&=\det(C\otimes \hat g-(c^{\top}c)\otimes(g^{vh}g^{hv})) \\
&=\det((C-c^{\top}c)\otimes \hat g+(c^{\top}c)\otimes(\hat g-g^{vh}g^{hv})).
\end{align*}
Since
\begin{align*}
& \det(C-c^{\top}c)=\det\bar C>0,\\
& \det \hat g>0,\\
& \det(c^{\top}c)\geq 0,\\ 
& \det(\hat g-g^{vh}g^{hv})=\det G>0,
& \end{align*}
it follows that matrices $(C-c^{\top}c)\otimes \hat g$ and $(c^{\top}c)\otimes(\hat g-g^{vh}g^{hv})$ are positive definite. Hence theirs sum is positive definite. 
\end{proof}

If $\bar C=I$ and $g_{TM}$ is the Sasaki metric, then we get Sasaki--Mok metric $\bar g^d$. 

Assume now $\bar C=I$ and $g_{TM}$ is a natural Riemannian metric on $TM$ \cite{ks,as} such that $g_{TM}(X^h,Y^h)=g(X,Y)$ and distributions $\mathcal{H}^{TM}$, $\mathcal{V}^{TM}$ are orthogonal. Hence, there are two smooth functions $\alpha,\beta:[0,\infty)\to\mathbb{R}$ such that
\begin{equation}\label{barg}
\begin{split}
& \bar g(X^h,Y^h)_u=g(X,Y), \\
& \bar g(X^h,Y^{v,i})_u=0,\\
& \bar g(X^{v,i},Y^{v,j})_u=0,\quad i\neq j,\\
& \begin{split}
\bar g(X^{v,i},Y^{v,i})_u &=\alpha(|F(u)|^2)g(X,Y)\\
&+\beta(|F(u)|^2)g(X,F(u))g(Y,F(u)).
\end{split}
\end{split}
\end{equation} 
The above Riemannian metric does not ''see'' the index $i$ of the distribution $\mathcal{V}^i$. Since all distributions $\mathcal{H}^{L(M)}, \mathcal{V}^1,\ldots,\mathcal{V}^n$ are orthogonal, it follows that we may put $F_i(u)=u_i$, that is consider a family of maps $F_1,\ldots,F_n$ rather than one map $F$, in the last condition, to obtain the positive definite bileinear form, hence the Riemannian metric,
\begin{equation}\label{bargnew}
\begin{split}
& \bar g(X^h,Y^h)_u=g(X,Y), \\
& \bar g(X^h,Y^{v,i})_u=0,\\
& \bar g(X^{v,i},Y^{v,j})_u=0,\quad i\neq j,\\
& \bar g(X^{v,i},Y^{v,i})_u=\alpha(|u_i|^2)g(X,Y)+\beta(|u_i|^2)g(X,u_i)g(Y,u_i).
\end{split}
\end{equation}  
We will write $\alpha_i$ and $\beta_i$ instead of $\alpha(|u_i|^2)$ and $\beta(|u_i|^2)$, respectively.

\section{Geometry of $\bar g$}
Let $(M,g)$ be a Riemanniann manifold, $(L(M),\bar g)$ its frame bundle equipped with the metric $\bar g$ of the form \eqref{bargnew}. Let $\overline\nabla$ and $\overline R$ denote the Levi--Civita conection and the curvature tensor of $\bar g$, respectively. 

We recall the identities concerning Lie bracket of horizontal and vertical wector fields \cite{ks0}
\begin{align}
& [X^h,Y^h]_u=[X,Y]_u^h-\sum_i \left(R(X,Y)u_i  \right)^{v,i}, \notag \\
& [X^h,Y^{v,i}]_u=\left( \nabla_X Y \right)_u^{v,i}, \label{liebracket} \\
& [X^{v,i},Y^{v,j}]_u=0. \notag
\end{align}
Moreover, in the local coordinates, for $X=\sum_i\xi_i\frac{\partial}{\partial x_i}$ we have
\begin{align}
& X^h(u^j_i)=-\sum_{a,b}\Gamma^j_{ab}u^a_i\xi_b \label{xhuij} \\
& X^{v,k}(u^j_i)=\xi_j\delta_{ik} \label{xvuij}
\end{align}
where $\Gamma^j_{ab}$ are Christoffel's symbols \cite{ks0}.

\begin{prop}
Connection $\overline\nabla$ satisfies the following relations
\begin{align*}
\left(\overline\nabla_{X^{h}}Y^{h}\right)_u &=\left( \nabla_X Y \right)^{h}_u-\frac{1}{2}\sum_i\left(R(X,Y)u_i\right)^{v,i}_u \\
\left(\overline\nabla_{X^h}Y^{v,i}\right)_u &=\frac{\alpha_i}{2}\left(R(u_i,Y)X\right)_u^h+\left( \nabla_XY \right)^{v,i}_u \\
\left(\overline\nabla_{X^{v,i}}Y^h\right)_u &=\frac{\alpha_i}{2}\left(R(u_i,X)Y\right)_u^h \\
\left(\overline\nabla_{X^{v,i}}Y^{v,j}\right)_u &=0\quad\textrm{for $i\neq j$},\\
\left(\overline\nabla_{X^{v,i}}Y^{v,i}\right)_u &=\frac{\alpha_i '}{\alpha_i}\left( g(X,u_i)Y^{v,i}+g(Y,u_i)X^{v,i} \right) \\
&+\left(\frac{\beta_i '\alpha_i-2\alpha_i '\beta_i}{\alpha_i(\alpha_i+|u_i|^2\beta_i)}g(X,u_i)g(Y,u_i)+\frac{\beta_i-\alpha_i '}{\alpha_i+|u_i|^2\beta_i}g(X,Y)\right)U^i,
\end{align*}
where $U^i_u=u_i^{v,i}$.
\end{prop}
\begin{proof}
Follows from the formula for the Levi--Civita connection 
\begin{align*}
2\bar g(\overline\nabla_AB,C) &=A\bar g(B,C)+B\bar g(A,C)-C\bar g(A,B)\\
&+\bar g([A,C],B)+\bar g([B,C],A)+\bar g([A,B],C)
\end{align*}
relations \eqref{liebracket} and the following equalities
\begin{align*}
& X^{v,i}_u(g(u_i,Y))=g(X,Y),\\
& X^{v,i}_u(|u_i|^2)=2(X,u_i),\\
& X^h_u(g(u_i,Y))=g(u_i,\nabla_XY).\qedhere
\end{align*}
\end{proof}

Before we compute the curvature tensor, we will need some formulas concerning the Levi--Civita connection $\overline\nabla$ of certain vector fields. 
\begin{lem}\label{lemat}
The following equalities hold
\begin{align*}
&\overline\nabla_{X^h}U^i=0,\\
&\overline\nabla_{X^{v,i}}U^j=0,\\
&\overline\nabla_{X^{v,i}}U^i=\frac{\alpha_i+|u_i|^2\alpha_i '}{\alpha_i}X^{v,i}+\frac{|u_i|^2(\alpha_i\beta_i '-\alpha_i '\beta_i)+\alpha_i\beta_i}{\alpha_i(\alpha_i+|u_i|^2\beta_i)}g(X,u_i)U^i. 
\end{align*}
and
\begin{equation*}
\overline\nabla_W (R(u_i,X)Y)^Q=\sum_j W(u^j_i)(R(u_i,X)Y)^Q+\sum_j u^j_i\overline\nabla_W (R(\frac{\partial}{\partial x_j},X)Y)^Q
\end{equation*}
for any $W\in TL(M)$ and $Q$ denoting the horiznotal or vertical lift.
\end{lem}
\begin{proof}
Folows by standard computations in local coordinates. 
\end{proof}

\begin{prop}\label{curv}
The curvature tensor $\bar R$ satisfies the following relations
\begin{align*}
\overline R(X^h,Y^h)Z^h &= (R(X,Y)Z)^h+\frac{1}{2}\sum_i(\nabla_Z R)(X,Y)u_i)^{v,i}\\
&-\frac{1}{4}\sum_i\alpha_i\left( R(u_i,R(Y,Z)u_i)X-R(u_i,R(X,Z)u_i)Y\right.\\
&\left.-2R(u_i,R(X,Y)u_i)Z \right)^h,
\end{align*}
\begin{align*}
\overline R(X^h,Y^h)Z^{v,i} &= (R(X,Y)Z)^{v,i}+\frac{\alpha_i}{2}\left( (\nabla_X R)(u_i,Z)Y-(\nabla_Y R)(u_i,Z)X \right)^h\\
&-\frac{\alpha_i}{4}\sum_j\left( R(X,R(u_i,Z)Y)u_j-R(Y,R(u_i,Z)X)u_j \right)^{v,j}\\
&+\frac{\alpha_i '}{\alpha_i}g(Z,u_i)(R(X,Y)u_i)^{v,i}-\frac{\beta_i-\alpha_i '}{\alpha_i+|u_i|^2\beta_i}g(R(X,Y)Z,u_i)U^i,
\end{align*}
\begin{align*}
\overline R(X^h,Y^{v,i})Z^h &=\frac{\alpha_i}{2}\left( (\nabla_XR)(u_i,Y)Z \right)^h-\frac{1}{2}(R(Z,X)Y)^{v,i}\\
&+\frac{\alpha_i '}{2\alpha_i}g(Y,u_i)(R(X,Z)u_i)^{v,i}-\frac{\alpha_i}{4}\sum_j(R(X,R(u_i,Y)Z)u_j)^{v,j}\\
&-\frac{\beta_i-\alpha_i '}{2(\alpha_i+|u_i|^2\beta_i)}g(R(X,Z)Y,u_i)U^i,
\end{align*}
\begin{align*}
\overline R(X^h,Y^{v,i})Z^{v,j} &= -\frac{\alpha_i\alpha_j}{4}\left( R(u_i,Y)R(u_j,Z)X \right)^h
\end{align*}
\begin{align*}
\overline R(X^h,Y^{v,i})Z^{v,i} &= \frac{\alpha_i '}{2}\left( g(Z,u_i)R(u_i,Y)X-g(Y,u_i)R(u_i,Z)X \right)^h\\
&-\frac{\alpha_i^2}{4}(R(u_i,Y)R(u_i,Z)X)^h-\frac{\alpha_i}{2}(R(Y,Z)X)^h
\end{align*}
\begin{align*}
\overline R(X^{v,i},Y^{v,i})Z^h &= \alpha_i(R(X,Y)Z)^h\\
&+\frac{\alpha_i^2}{4}\left( R(u_i,X)R(u_i,Y)Z-R(u_i,Y)R(u_i,X)Z \right)^h \\
&+\alpha_i '(g(X,u_i)(R(u_i,Y)Z)^h-g(Y,u_i)(R(u_i,X)Z)^h)
\end{align*}
\begin{align*}
\overline R(X^{v,i},Y^{v,j})Z^h &= \frac{\alpha_i\alpha_j}{4}\left( R(u_i,X)R(u_j,Y)Z-R(u_j,Y)R(u_i,X)Z \right)^h
\end{align*}
\begin{align*}
\overline R(X^{v,i},Y^{v,i})Z^{v,i}) &=C_i(g(X,u_i)g(Y,Z)-g(Y,u_i)g(X,Z))U^i\\
&+(A_ig(Y,u_i)g(Z,u_i)+B_ig(Y,Z))X^{v,i}\\
&-(A_ig(X,u_i)g(Z,u_i)+B_ig(X,Z))Y^{v,i}
\end{align*}
\begin{align*}
\overline R(X^{v,i},Y^{v,j})Z^{v,k} &=0 \quad\textrm{if $\sharp\{i,j,k\}>1$}
\end{align*}
where
\begin{align*}
A_i &=\frac{3(\alpha_i ')^2-2\alpha_i\alpha_i ''}{\alpha_i^2}+\frac{(\alpha_i\beta_i '-2\alpha_i '\beta_i)(\alpha_i+|u_i|^2\alpha_i ')}{\alpha_i^2(\alpha_i+|u_i|^2\beta_i)},\\
B_i &=\frac{\alpha_i\beta_i-2\alpha_i\alpha_i '-(\alpha_i ')^2|u_i|^2}{\alpha_i(\alpha_i+|u_i|^2\beta_i)},\\
C_i &=-\frac{2\alpha_i ''}{\alpha_i+|u_i|^2\beta_i}+\frac{3\alpha_i(\alpha_i ')^2+2(\alpha_i ')^2\beta_i|u_i|^2+\alpha_i^2\beta_i '-\alpha_i\beta_i^2+\alpha_i\alpha_i '\beta_i '|u_i|^2}{\alpha_i(\alpha_i+|u_i|^2\beta_i)^2)}
\end{align*}
\end{prop}
\begin{proof}
Follows from the characterization of the Levi--Civita connection $\overline{\nabla}$ and Lemma \ref{lemat}.
\end{proof}

\begin{rem}
Notice that
\begin{equation*}
A_i\alpha_i-B\beta_i=C_i(\alpha_i+|u_i|^2\beta_i),
\end{equation*}
which is equivalent to the condition
\begin{equation*}
\bar g(\bar R(X^{v,i},Y^{v,i})Z^{v,i},W^{v,i})=\bar g(\bar R(Z^{v,i},W^{v,i})X^{v,i},Y^{v,i}).
\end{equation*}
\end{rem}

\begin{cor}\label{seccurv}
Let $X,Y$ be two orthonormal vectors in the tangent space $T_xM$. Then the scalar curvature $\overline K$ of $(L(M),\bar g)$ and $K$ of $(M,g)$ are related as follows
\begin{align*}
& \overline K(X^h,Y^h)=K(X,Y)-\frac{3}{4}\sum_i\alpha_i|R(X,Y)u_i|^2,\\
& \overline K(X^h,Y^{v,i})=\frac{\alpha_i^2}{4(\alpha_i+\beta_ig(Y,u_i)^2)}|R(u_i,Y)X|^2,\\
& \overline K(X^{v,i},Y^{v,i})=\frac{A_i(g(X,u_i)^2+g(Y,u_i)^2)
+B_i}{\alpha_i+\beta_i(g(X,u_i)^2+g(Y,u_i)^2)},\\
& \overline K(X^{v,i},Y^{v,j})=0\quad\textrm{for $i\neq j$}.
\end{align*}
In particular, if $(M,g)$ is of constant sectional curvature $\kappa$, then
\begin{align*}
& \overline K(X^h,Y^h)=\kappa-\frac{3}{4}\kappa^2\sum_i\alpha_i(g(X,u_i)^2+g(Y,u_i)^2), \\
& \overline K(X^h,Y^{v,i})=\frac{\kappa^2\alpha_i^2g(X,u_i)^2}{4(\alpha_i+\beta_ig(Y,u_i))}\geq 0.
\end{align*}
If, moreover, $\sum_i\alpha_i(t_i)t_i<\frac{4}{3\kappa}$ for all $t_i>0$, then $\overline K(X^h,Y^h)>0$.
\end{cor}
\begin{proof}
The formula for $\overline K$ follows by Proposition \ref{curv}. Since $g(X,u_i)^2+g(Y,u_i)^2\leq |u_i|^2$, hence, if $(M,g)$ is of constant sectional curvature and $\sum_i\alpha_i(t_i)t_i<\frac{4}{3\kappa}$, then 
\begin{equation*}
\overline K(X^h,Y^h)\geq \kappa-\frac{3}{4}\kappa^2\sum_i\alpha_i|u_i|^2>0.
\end{equation*}
\end{proof}

\begin{cor}
The scalar curvature $\bar s$ of $(L(M),\bar g)$ at $u\in L(M)$ is of the form
\begin{align*}
\bar s &=s-\frac{1}{4}\sum_{i,j,k}\alpha_k|R(e_i,e_j)u_k|^2\\
&+\sum_k\left(n(n-1)\frac{B_k}{\alpha_k}
+\frac{2(nA_k\alpha_k-B_k\beta_k)}{\alpha_k^2}|u_k|^2+\frac{(n+3)C_k\beta_k}{\alpha_k^2}|u_k|^4\right.\\
&\left.+\frac{(n-1)\beta_k(B_k(2\alpha_k+\beta_k)+A_k\alpha_k)}{\alpha_k^2(\alpha_k+|u_k|^2\beta_k)}
+\frac{2C_k\beta_k^2}{\alpha_k^2(\alpha_k+|u_k|^2\beta_k)}|u_k|^6\right),
\end{align*}
where $s$ is the scalar curvature of $(M,g)$ and $e_1,\ldots,e_n$ is an orthonormal basis in $T_xM$, $\pi_{L(M)}(u)=x$.
\end{cor}
\begin{proof}
Fix $u\in L(M)$ and let $e_1,\ldots,e_n$ be an orthonormal basis in $T_xM$, $\pi_{L(M)}(u)=x$. Then \eqref{basis} forms a basis of $T_uL(M)$. Put
\begin{equation*}
\bar g^k_{ij}=\bar g(e_i^{v,k},e_j^{v,k})=\alpha_k\delta_{ij}+\beta_kg(e_i,u_k)g(e_j,u_k).
\end{equation*}
The inverse matrix $(\bar g^{ij}_k)$ to $(\bar g^k_{ij})$ is of the form
\begin{equation*}
\bar g^{ij}_k=\frac{1}{\alpha_k}\delta_{ij}-\frac{\beta_k}{\alpha_k(\alpha_k+|u_k|^2\beta_k)}g(e_i,u_k)g(e_j,u_k).
\end{equation*} 
Hence
\begin{align*}
\bar s &=\sum_{i,j}\bar g(\bar R(e_i^h,e_j^h)e_j^h,e_i^h)+2\sum_{i,j,l,k}\bar g^{jl}_k\bar g(\bar R(e_i^h,e_j^{v,k})e_l^{v,k},e_i^h)\\
&+\sum_{i,j,k,l,p}\bar g^{ip}_k\bar g^{jl}_k\bar g(\bar R(e_i^{v,k},e_j^{v,k})e_l^{v,k},e_p^{v,k})\\
\end{align*}
Follows now from Proposition \ref{curv} and the equality
\begin{equation*}
\sum_{i,j}|R(e_i,e_j)u_k|^2=\sum_{i,j}|R(u_k,e_j)e_i)|^2.\qedhere
\end{equation*}
\end{proof}

In the case of a Cheeger--Gromoll type metric we have:
\begin{cor}
Assume
\begin{equation*}
\alpha_i(t)=\beta_i(t)=\frac{1}{1+t},\quad t>0.
\end{equation*}
Then
\begin{align*}
\overline K(X^{v,i},Y^{v,i})=\frac{-t_i(g(X,u_i)^2+g(Y,u_i)^2)+t_i^2+3t_i+3}{(1+t_i)^2(1+g(X,u_i)^2+g(Y,u_i)^2)},
\end{align*}
where $t_i=|u_i|^2$. In particular, if $(M,g)$ is of constant sectional curvature $0<\kappa<\frac{4}{3n}$, then sectional curvature $\overline K$ is nonnegative.
\end{cor}

\begin{proof}
We have
\begin{equation*}
\sum_i\alpha_i(t_i)t_i=\sum_i\frac{t_i}{1+t_i}<\frac{4}{3\kappa}\quad\textrm{for all $t_i>0$}
\end{equation*}
if and only if $0<\kappa<\frac{4}{3n}$. Hence, by Corollary \ref{seccurv} $\overline K(X^h,Y^h)\geq 0$ for $X,Y\in T_xM$ unit and orthogonal. Moreover, $g(X,u_i)^2+(Y,u_i)^2\leq |u_i|^2=t_i$. Thus
\begin{equation*}
\bar K(X^{v,i},Y^{v,i})\geq \frac{-t_i^2+t_i^2+3t_i+3}{t_i(1+t_i)^2}=\frac{3}{t_i(t_i+1)}>0.\qedhere
\end{equation*}
\end{proof}

\end{document}